\documentclass[11pt]{article}
\usepackage{amsmath,amssymb,amsthm,xcolor,graphicx,bbm}
\usepackage{accents}
\usepackage[unicode,breaklinks=true,colorlinks=true,citecolor = {magenta}]{hyperref}
\usepackage{esint}  %
\usepackage{listings}  %

\usepackage[top=1in, bottom=1in, left=1.25in, left=1.25in]{geometry}

\usepackage{titletoc}
\titlecontents{section}[1.5em]{\vspace{-2pt}}%
{\thecontentslabel\enspace\enspace}%
{}%
{\titlerule*[0.5pc]{.}\contentspage}
\titlecontents{subsection}[3em]{\vspace{-2.5pt}}%
{\thecontentslabel\enspace\enspace}%
{}%
{\titlerule*[0.5pc]{.}\contentspage}%

\numberwithin{equation}{section}
\newtheorem{theorem}{Theorem}[section]

\newtheorem{lemma}[theorem]{Lemma}
\newtheorem{proposition}[theorem]{Proposition}

\newtheorem{conjecture}[theorem]{Conjecture}

\theoremstyle{remark}
\newtheorem{remark}[theorem]{Remark}
\newcommand{\bke}[1]{\left ( #1 \right )}

\newcommand{\bket}[1]{\left \{ #1 \right \}}
\newcommand{\norm}[1]{\left  \| #1  \right \|}
\newcommand{\bka}[1]{{\langle #1 \rangle}}

\newcommand\ga{\gamma}
\newcommand\de{\delta}
\newcommand\ep{\epsilon}

\newcommand\e {\varepsilon}

\newcommand\la{\lambda}
\newcommand\si{\sigma}

\newcommand\De{\Delta}

\newcommand\Om{\Omega}
\newcommand{\R}{\mathbb{R}}

\newcommand{\FF}{\mathbb{F}}
\newcommand{\GG}{\mathbb{G}}

\newcommand{\NN}{\mathbb{N}}

\newcommand{\cW}{\mathcal{W}}

\renewcommand{\div}{\mathop{\rm div}\nolimits}

\newcommand{\pd}{\partial}
\newcommand{\nb}{\nabla}

\newcommand{\td}{\tilde}

\newcommand{\oo}{\infty}

\newcommand{\EQ}[1]{\begin{equation} #1 \end{equation}}
\newcommand{\Eq}[1]{\begin{equation*} #1 \end{equation*}}
\newcommand{\EQS}[1]{\begin{equation}\begin{split} #1 \end{split}\end{equation}}
\newcommand{\EQN}[1]{\begin{equation*}\begin{split} #1 \end{split}\end{equation*}}

\newcommand{\EN}[1]{\begin{enumerate} #1 \end{enumerate}}

\usepackage{mathtools} %

\newcommand{\loc}{\mathrm{loc}}

\begin{document}
\title{Existence and regularity for perturbed Stokes system with critical drift}
\author{Misha Chernobai and Tai-Peng Tsai}
\date{}
\maketitle
\begin{abstract}
We consider the existence and $L^q$ gradient estimates for perturbed Stokes systems with divergence-free critical drift in a bounded Lipschitz domain in $\mathbb{R}^n$, $n \ge 3$. The first two results assume the drift is either in $L^n$
or sufficiently small in weak $L^n$. The third result assumes the drift is in weak $L^n$ without smallness, and obtain results for $q$ close to 2.

\medskip

\emph{2020 Mathematics Subject Classifications}: 35J47, 35Q35
\end{abstract}

\section{Introduction}

Let $\Om$ be a bounded domain in $\R^n$, $n\ge 3$. We consider the existence and gradient estimates of solutions $(u,\pi):\Om \to \R^n \times \R$ of the following perturbed Stokes system
\EQ{\label{drifteq}
-\Delta u +b\cdot\nabla u+\nabla \pi=\div \Bbb G,\quad \div u=0, \quad u|_{\pd \Om}=0,
}
with given  $b$ in the critical space, $b\in L^{n,\oo}(\Omega)^n$, $\div b=0$, and $\Bbb G\in L^q(\Om)^{n\times n}$, $1 <q<\oo$. Here
$L^{q,r}(\Om)$ denotes Lorentz spaces. We use the convention
$(\div \Bbb G)_i=\pd_j \GG_{ji}$ and $(\nb \zeta)_{ji} = \pd_j \zeta_i$ so that
\EQ{\label{1.2}
\De u = \div (\nb u),\quad
b\cdot\nabla u = \div (b\otimes u), \quad \int_\Om  \div \GG \cdot \zeta = -
\int_\Om  \GG:\nb \zeta
}
when $\div b=0$ and boundary term vanishes.
Let
\begin{align*}
C^{\infty}_{c,\si}(\Omega) &= \big\{ u \in C^{\infty}_c(\Omega)^n:\  \div u=0\big\},\\
W^{1,q}_{0,\si}(\Omega) &= \big\{ u \in W^{1,q}_0(\Omega)^n:\  \div u=0\big\},\\ \quad L^q_0(\Om)&=\bket{\pi \in L^q(\Om):\  \textstyle \int_\Om \pi \,dx=0}.
\end{align*}

We say $(u,\pi)$ is a \emph{solution pair} of \eqref{drifteq} if $u\in W^{1,q}_{0,\si}(\Omega)$, $\pi\in L^q_0(\Omega)$, and $(u,\pi)$ satisfies  \eqref{drifteq} in distributional sense, i.e., (using \eqref{1.2}),
\EQ{\label{soln-pair}
\int_{\Om} (\nb u - b \otimes u): \nb \zeta - \pi \div \zeta  = -\int_\Om \GG : \nb \zeta,\quad
\forall \zeta\in C^{\infty}_{c,\si}(\Omega).
}
We say $u$ is a \emph{$q$-weak solution} of \eqref{drifteq} if $u\in W^{1,q}_{0,\si}(\Omega)$ and $u$ satisfies \eqref{drifteq} in weak sense, i.e.,
\EQ{\label{weak-soln}
\int_{\Om} (\nb u - b \otimes u): \nb \zeta   = -\int_\Om \GG : \nb \zeta,\quad
\forall \zeta\in C^{\infty}_{c,\si}(\Omega).
}
We simply say $u$ is a \emph{weak solution} if $q=2$.
We say $u$ is a \emph{very weak solution} of \eqref{drifteq} if $u\in L^{n',1}_\loc(\Omega)$, $\div u=0$, and $u$ satisfies  \eqref{drifteq} in the sense
\EQ{\label{very-weak-soln}
\int_{\Om} - u\cdot \De \zeta - ( b \otimes u): \nb \zeta   = -\int_\Om \GG : \nb \zeta,\quad
\forall \zeta\in C^\infty_{c,\si}(\Omega).
}
Note that this definition of very weak solutions using test fields $ \zeta\in C^\infty_{c,\si}(\Om)$ gives no info on the boundary value of $u$, unlike another definition using $ \zeta\in C^2_0(\Om)$, see \cite[IV, Notes]{Galdi}. 

\smallskip

Our study is motivated by the following open problems. Let $B_r=\{ x \in \R^n: |x|<r\}$ denote the ball of radius $r$ centered at the origin.

\begin{conjecture} \label{conj1}
If a vector field $u$ in $\R^3\setminus \{ 0\}$
satisfies $|u(x)| \le \frac C{|x|}$ for all $x \not=0\in \R^3$ for a large constant $C$, and $u$ is a very weak solution of  the stationary incompressible Navier-Stokes equations in $\R^3\setminus \{ 0\}$, then $u$ is one of the Landau solutions which are minus-one homogeneous and axisymmetric without swirl.
\end{conjecture}

Here ``very weak solution'' is in the sense of \eqref{very-weak-soln} with $b=u$.
See \cite{Hishida, JST-survey,Tsai-book} for Landau solutions and the background of this problem. Conjecture \ref{conj1} was stated by Sverak in \cite{Sverak2011}, as well as Miura-Tsai in \cite{MT12}  and Kwon-Tsai \cite{KT}, also for the related studies see \cite{Korolev-Sverak,Li-Li-Yan3}
for related study.
Relevant to Conjecture \ref{conj1} is the following question.

\begin{conjecture} \label{conj2}
Let $n=3$, $b\in L^{3,\infty}(B_2)$, and $\div b=0$. If a vector field $u \in W^{1,(\frac32,\infty)}(B_2)$  is a very weak solution of \eqref{drifteq} with zero force in $B_2$, i.e., it satisfies
\[
\int_{B_2} (\nb u - b \otimes u): \nb \zeta  =0,\quad
\forall \zeta\in C^\infty_{c,\si}(B_2),
\]
then $u \in L^\infty(B_1)$.
\end{conjecture}

The condition $u \in W^{1,(3/2,\infty)}(B_2)$ is because a solution of Conjecture \ref{conj1} satisfying $|u(x)| \le \frac C{|x|}$ is shown to satisfy $|\nb^k u(x)| \le C_k |x|^{-1-k}$, $k \in \NN$, see \cite{Sverak-Tsai}. When $b=u$, Conjecture \ref{conj2} has been proved by
Kim \& Kozono \cite{KiKo} when $u$ is sufficiently small in $L^{3,\infty}(B_2)$. {For higher dimensional results see e.g.~\cite{BGLWX}.}

To study these conjectures, it seems necessary to have a better understanding of the regularity problem of vector solutions of the perturbed Stokes system \eqref{drifteq}.

We first review what is known for the corresponding equations for \emph{scalar} functions.
In the literature, we allow $\div b \not =0$, and consider weak solutions  of
\EQ{
-\Delta u +\nabla \cdot(ub)=\div F, \quad u|_{\pd \Om}=0,
\label{scalareq-div}
}
and its dual problem
\EQ{
-\Delta v -b\cdot\nabla v=\div G, \quad v|_{\pd \Om}=0.
\label{scalareq-conv}
}
Existence, uniqueness, and regularity of weak solutions in $W^{1,q}(\Omega)$ or  $W^{2,q}(\Omega)$, $1<q<\infty$, of \eqref{scalareq-div} and \eqref{scalareq-conv} have been well known for sufficiently regular $b$;
for instance, see  \cite[Theorem 9.15]{GT}  for $b\in L^\infty(\Omega)$, and \cite[Chap.~III, Theorem 15.1]{LU} for more general $b \in L^m(\Om; \R^n )$, $m>n$. This is called the ``\emph{subcritical case}'', in which the lower order terms ($\nabla \cdot(vb)$ and $b\cdot\nabla v$) may be treated as perturbations of  the leading term $-\De v$.
In the ``\emph{critical case}'',  the coefficient $b$  belongs to \emph{critical}  spaces, that is,
\[%
b \in L^n(\Om ; \R^n), \quad \text{or}\quad
b \in L^{n,\infty}(\Om; \R^n),
\]
which prevents us from treating the lower order terms as  perturbations. For general $b \in L^{n,\infty}$, \eqref{scalareq-div} may have no solution, and \eqref{scalareq-conv} may have no uniqueness,  see e.g.~\cite[Example 1.1]{KPT}. Lower order terms with critical coefficients can be “controlled” in a few cases. The first case is when $b$ is small in $L^{n,\infty}$. The second case is when $b\in L^n$  or $b \in L^{n,r}$, $r<\infty$, so that the norms become small in sufficiently small balls although they may be large in the entire domain $\Om$. In both case we gain certain smallness of the drift term.
The third is case when $\div b=0$ or $\div b$ has a fixed sign. In this case, the drift term may become harmless in De Giorgi or Nash type estimates after integration by parts.  See the vast references of \cite{KPT}, in which an attempt is made to combine these cases. In general case, when drift is not small or divergence free there are similar results under additional assumptions on the structure of the drift, see \cite{CS20,CS24}.

\smallskip

When we consider the perturbed Stokes system \eqref{drifteq} {with critical drift,} even less is known. The presence of the pressure term $\nb p$ adds new difficulties to the methods used to study the scalar equations
\eqref{scalareq-div} and \eqref{scalareq-conv}.

In this note will present a few results on the perturbed Stokes system \eqref{drifteq} assuming $\div b=0$ (so that we do not need to study its dual system, which is now equivalent to
\eqref{drifteq}).
Our first two results, {Proposition \ref{main1} and Theorem \ref{main2},} give gradient estimates of \eqref{drifteq} assuming smallness conditions on the drift $b$ in critical spaces. The condition $b\in L^n(\Om)^n$ in {Proposition \ref{main1}} can be considered a smallness condition because it
implies $$\lim_{r \to 0_+} \sup_{x_0 \in \Om}\norm{b}_{L^n(\Om \cap B_r(x_0))}=0.$$ They are similar to the case of the scalar equations, and should be no surprise to experts.
\begin{proposition}\label{main1}
Let $1<q<\infty$ and $\Om$ be a bounded Lipschitz domain in $\R^n$, $n\ge 3$,
with sufficiently small Lipschitz constant $L>0$ $($i.e., $L \le L_0$ where $L_0=L_0(n,q)>0)$.
Assume $b\in L^n(\Omega)^n$, $\div b=0$.  Then for any $\GG\in L^q(\Om)^{n\times n}$, , there exists a unique solution pair $u\in W^{1,q}_{0,\si}(\Omega)$ and $\pi\in L^q(\Omega)$ of \eqref{drifteq} such that
$\int_\Om \pi\,dx=0$ and
\EQ{\label{thm1.3-eq}
\|u\|_{ W^{1,q}(\Omega)}+\|\pi\|_{L^q(\Omega)}\le C\|\Bbb G\|_{L^q(\Omega)}.
}
\end{proposition}
We may simply assume $\Om$ is a bounded $C^1$ domain, as for
a bounded $C^1$ domain and arbitrary small constant $\lambda>0$, each point $x$ on  $\pd\Om$ has a nieghborhood $U_x$ such that $\pd \Om \cap U_x$ is the graph of a Lipschitz function with a Lipschitz constant less than $\lambda$.

{The boundary value problem of the perturbed Stokes system \eqref{drifteq} for critical drift $b\in L^3(\Om)$, $\Om \subset \R^3$, has previously been studied;
See Section 4 of Kim \cite{Kim09}, Theorem 6.1 in Dindoš and Mitrea \cite{DinMit}, Theorem 18 in Choe \& Kim \cite{ChoeKim}, and Section 5 of
Amrouche \& Rodríguez-Bellido \cite{AmRB}. They contain estimates similar to \eqref{thm1.3-eq}.
For more general results on scalar equations \eqref{scalareq-div} and \eqref{scalareq-conv} with drift}  $b \in L^n$, see
Kim \& Kim \cite{KK}.

\begin{theorem}\label{main2}
Let $1<q<\infty$ and $\Om$ be a bounded Lipschitz domain in $\R^n$, $n\ge 3$,
with sufficiently small Lipschitz constant $L>0$ $($i.e., $L \le L_0$ where $L_0=L_0(n,q)>0)$.
There is $\e(q,\Om)>0$ such that, if
$b\in L^{n,\infty}(\Omega)^n$, $\div b=0$, and $\norm{b}_{L^{n,\infty}(\Omega)} \le \e(q,\Om)$, then for any $\GG\in L^q(\Om)^{n\times n}$,
there exists a unique solution pair $u\in W^{1,q}_{0,\si}(\Omega)$ and $\pi\in L^q(\Omega)$ of \eqref{drifteq} such that $\int_\Om \pi\,dx=0$ and
\EQ{
\|u\|_{ W^{1,q}(\Omega)}+\|\pi\|_{L^q(\Omega)}\le C\|\Bbb G\|_{L^q(\Omega)}.
}
\end{theorem}

Theorem \ref{main2} is similar to {Proposition \ref{main1}} with the smallness directly assumed on the $L^{n,\infty}$-norm of $b$. Again, we may simply assume $\Om$ is a bounded $C^1$ domain.
Theorem \ref{main2} may hold for an arbitrary bounded Lipschitz domain in dimension 3, without assuming small Lipschitz constant, but only for $q$ close to 2, if one uses Dindo\v s-Mitrea \cite[Theorem 5.6]{DinMit} instead of Theorem \ref{th:GSS}. We do not pursue it here.

{A related result is Dong \& Phan \cite[Theorem 1.9]{DongPhan}, which is concerned with \emph{non-stationary} perturbed Stokes systems. For divergence free drift  %
sufficiently small in BMO$^{-1}$, which is also critical, they obtain an \emph{interior} gradient estimate. It does not include the boundary estimate due to difficulties with the time derivative.
In the time-independent case, the same method should work, and could be combined with the argument in Dong \& Kim \cite{DongKim} to give the boundary estimate.}
For corresponding results on scalar equations with drift coefficient $b \in L^{n,\infty}$, see
Kim \& Tsai \cite{KiTs}.

\medskip

Our third result, {which is the most interesting in this paper,} does not assume smallness of the drift $b$, but it only gives gradient estimates for exponent $p$ slightly larger than 2, with $p-2$ bounded above by a constant depending on the size of $b$ in $L^{n,\infty}(\Om)$. It is based on Caccioppoli inequality approach as in the classical work of Gehring \cite{Gehring} and Giaquinta-Modica \cite{GM79a,GM82}, which has been adapted by Kwon \cite{Kwon_1} for scalar equations with critical drifts.  However, to bound the pressure in \eqref{drifteq} at the presence of the critical drift term, we will use Wolf's local pressure projection \cite{Wolf15,Wolf17},
which is the key new ingredient compared to  \cite{Kwon_1}.
We do not need the Lipschitz constant $L$ to be small, but we do need $L<1/2$.

\begin{theorem}\label{main3}
Let $\Om$ be a bounded Lipschitz domain in $\R^n$, $n\ge 3$, with Lipschitz constant $0<L<1/2$.
Assume $b\in L^{n,\infty}(\Omega)^n$, $\div b=0$.
Then there exists $p_0=p_0(\Omega,\|b\|_{L^{n,\infty}(\Omega)})>2$ such that {for each} $q\in(p_0',p_0)$ and $\GG\in L^{q}(\Omega)^{n\times n}$,  there exists a unique solution pair $u\in W^{1,q}_{0,\si}(\Omega)$ and $\pi\in L^q(\Omega)$ of \eqref{drifteq} such that $\int_\Om \pi dx=0$. Moreover, we have
\EQ{
\|u\|_{ W^{1,q}(\Omega)}+\|\pi\|_{L^q(\Omega)}\le C\|\GG\|_{L^{q}(\Omega)},
}
with $C=C(\Omega,\|b\|_{L^{n,\infty}(\Omega)})$.
\end{theorem}

Similar to scalar equations \eqref{scalareq-div} and
\eqref{scalareq-conv}, the key in the proof of {Theorem \ref{main3}} is to establish an a priori bound. The new difficulty for the perturbed Stokes system \eqref{drifteq} is the pressure term $\nb\pi$, which cannot be eliminated or hidden as $b \cdot \nb u$, when we test it with a test function of the form $f(|u|)u \zeta$ where $f$ is a Lipschitz continuous function and $\zeta$ is a smooth cut-off function. This has been the main obstacle for us to apply the methods of De Giorgi and Moser which were successfully used in \cite{KiTs,KPT}. However, we are able to deal with the pressure term in the approach of Gehring \cite{Gehring}, Giaquinta-Modica  \cite{GM79a,GM82} and Kwon \cite{Kwon_1}, using the new ingredient of the local pressure projection of Wolf \cite{Wolf15,Wolf17}. This results in Theorem \ref{main3}.

By a similar proof as Theorem \ref{main3}, we can obtain the following proposition for the \emph{scalar equations} \eqref{scalareq-div} and
\eqref{scalareq-conv}. {Instead of divergence free, we assume $\div b\ge0$.}

\begin{proposition}[Scalar equations]\label{main4}
Let $\Om$ be a bounded Lipschitz domain in $\R^n$, $n\ge 3$.
Assume $b\in L^{n,\infty}(\Omega)^n$, $\div b\ge 0$.
Then there exists $p_0=p_0(\Omega,\|b\|_{L^{n,\infty}(\Omega)})>2$ such that

\textup{(a)} For any $q\in[2,p_0]$ and $G\in L^{q}(\Omega)^{n}$,  there exists a unique solution $v\in W^{1,q}_{0}(\Omega)$ of \eqref{scalareq-conv}. Moreover, we have
\EQ{\label{v.est}
\|v\|_{ W^{1,q}(\Omega)}\le C\|G\|_{L^{q}(\Omega)}.
}

\textup{(b)} For any $q\in[p_0',2)$ and $F\in L^{q}(\Omega)^{n}$,  there exists a unique solution $u\in W^{1,q}_{0}(\Omega)$ of \eqref{scalareq-div}. Moreover, we have
\EQ{\label{u.est}
\|u\|_{ W^{1,q}(\Omega)}\le C\|F\|_{L^{q}(\Omega)}.
}
\end{proposition}

Part (a) is an optimized version for $n \ge 3$ of Kwon \cite[Theorem 1.1]{Kwon_1}, which states for $n \ge2$ and $q>2$, the existence of $0<\e<q-2$ and solution $v$ such that $\|v\|_{ W^{1,2+\e}(\Omega)}\le C\|G\|_{L^{q}(\Omega)}$.
The main case of the results of \cite{Kwon_1} is when $n=2$, which is
very subtle and hard to improve.

On the other hand, Proposition \ref{main4} is a weaker result than Kim-Tsai \cite[Theorem 2.1]{KiTs}, which is why we only call it a proposition:
In Proposition \ref{main4}, $p_0$ is only slightly larger than  $2$, but $p_0=n$
in \cite[Theorem 2.1]{KiTs}. The only advantage of Proposition \ref{main4} is that
$\Om$ does not need be $C^1$, and $\div b$ need not be in $L^{n/2,\infty}(\Om)$.
The proofs are different:
The first step in the proof  of \cite[Theorem 2.1]{KiTs} is a priori estimate in $L^{p^*}$, $2<p<n$, for solution $u$ of \eqref{scalareq-div} with $F \in L^p$ using De Giorgi - Moser type method, from which one gets $W^{1, (p,\infty)}$ estimates by Calderon-Zygmund estimates in Lorentz spaces.
In contrast, the first step in the proof of Proposition \ref{main4} is a reversed H\"older inequality for solution $v$ of \eqref{scalareq-conv} in $L^2$, shown in Kwon \cite{Kwon_1}.

\medskip

The rest of the paper is organized as follows. In Section \ref{sec2} we give a few preliminary results, including H\"older and Sobolev inequalities in Lorentz spaces, drift term estimates, and $L^q$ estimates of the Stokes system. In Section \ref{sec3} we first prove a priori $W^{1,q}$ estimates, and then prove {Proposition \ref{main1} and Theorem \ref{main2}}.
In Section \ref{sec4} we use an explicit Bogovskii's map to prove uniform $L^2$-estimates of Wolf's local pressure projections on our regions.
In Section \ref{sec5} we prove a priori $W^{1,q}$ estimates under the assumptions of Theorem \ref{main3}.
In Section \ref{sec6} we prove Theorem \ref{main3}.
In Section \ref{sec7} we prove Proposition \ref{main4}.
\section{Preliminaries}\label{sec2}
In this section we give a few preliminary results. For a bounded open set $E\subset \R^n$, we denote
\[
(f)_E  = \frac 1{|E|} \int_E f\,dx = \fint_E f\,dx,
\]
\[
\Om_r(x_0)=\Om_{x_0,r} = \Om \cap B_r(x_0).
\]
We sometimes abbreviate  $\norm{f}_{L^p(E)}$ as $\norm{f}_{p,E}$ for the $L^p$-norm in $E$.
As $\Om$ is bounded, we can use the following norm for $W^{1,-q}(\Om)^m$, $1<q<\infty$, $m\ge1$,
\EQ{
\norm{f}_{W^{1,-q}(\Om)^m} = \sup _{u \in W^{1,q'}_0(\Om)^m,\,\norm{\nb u}_{L^{q'}}(\Om)\le1}
\bka{f, u}.
}

We now recall H\"older and
Sobolev inequalities in Lorentz spaces.
The following is the H\"older inequality in Lorentz spaces, essentially due
to R. O’Neil \cite{ON}. See \cite[Lemma 3.1]{KPT} for the cases $p \le 1$ or $q< 1$.

\begin{lemma}[H\"older inequality in Lorentz spaces]\label{Holder}
Let $\Om$ be any domain in $\R^n$.
Let $0< p,p_1,p_2< \infty$ and $0< q,q_1,q_2\le\infty$ satisfy
$$
\frac{1}{p}=\frac{1}{p_1}+\frac{1}{p_2}\quad\text{and}\quad \frac{1}{q} \le  \frac{1}{q_1}+\frac{1}{q_2}.
$$
Then there is a constant $C=C(  p_1 , p_2 ,   q_1 , q_2 ,q )>0$ such that
$$
\|fg\|_{p,q} \le C  \|f\|_{p_1,q_1} \|g\|_{p_2,q_2}
$$
for all $f \in L^{p_1,q_1}(\Omega ) $ and $g \in L^{p_2,q_2}(\Omega )$.
\end{lemma}

The following is the Sobolev inequality in Lorentz spaces. It follows from the
Sobolev inequality in Lorentz spaces in the whole $\R^n$ (see \cite[Remark 7.29]{AdamsFournier} and \cite{PO}), and the extension theorem from $W^{1,q} (\Om )$ to $W^{1,q} (\R^n)$  (see \cite[Theorem 5.28]{AdamsFournier}).

\begin{lemma}[Sobolev inequality in Lorentz spaces]\label{Sobolev}
Let $\Om$ be a bounded Lipschitz domain in $\R^n$, $n\ge 2$.
For $1< q<n$, there is a constant $C=C(n,q,\Om)>0$ such that
\[
\|u \|_{L^{q^* , q}(\Om )} \le C  \| u \|_{W^{1,q}(\Om)},\quad \forall u \in W^{1,q} (\Om ).
\]
Moreover, if $u$ satisfies zero boundary condition then $C$ is independent of $\Omega$.
\end{lemma}

This lemma does not include $q=1$ as its proof uses the extension theorem \cite[Theorem 5.28]{AdamsFournier} which is based on the Calderon-Zygmund theory. We can include $q=1$ if $\Om$ is $C^1$.
The combination of the above 2 lemmas shows $b u \in L^q(\Om)$, $1\le q<n$, if $b \in L^{n,\infty}$ and $u \in W^{1,q}(\Om)$,
\EQ{\label{eq2.1}
\norm{bu}_q \le C \norm{b}_{n,\infty} \norm{u}_{q^*, q} \le C  \norm{b}_{n,\infty} \| u \|_{W^{1,q}(\Om)},
}
and allows us to take test fields $ \zeta\in W^{1,q'}_0$
in \eqref{soln-pair} and \eqref{weak-soln}.

In the following two lemmas the drift term $b\cdot\nabla u$ is estimated in $W^{-1,q}(\Omega)$. These estimates are classical and can be obtained, e.g., from \cite[Lemma 3.5]{KPT} in which $\div b\neq 0$. We give a streamlined proof under our simpler setting with $\div b=0$.%

\begin{lemma}\label{drift_est}
Let $\Om$ be a bounded Lipschitz domain in $\R^n$, $n\ge 3$. For any
$b\in L^n(\Omega)$, $\div b=0$, and $u\in W^{1,q}(\Omega)$, $1\le q<\oo$, we have
\EQ{\label{eq2.4}
b\cdot\nabla u\in W^{-1,q}(\Omega)~\text{and}~\|b\cdot \nabla u\|_{W^{-1,q}(\Omega)}\le C_q\|b\|_{n,\Omega}\|u\|_{W^{1,q}(\Omega)}.
}
Moreover, for each $\varepsilon>0$ there exists $C=C(\varepsilon,p,b,\Om)$ such that
\EQ{\label{eq2.5}
\|b\cdot\nabla u\|_{W^{-1,q}(\Omega)}\le \varepsilon \|u\|_{1,q;\Omega}+C \|u\|_{q,\Omega}
}
\end{lemma}
\begin{proof}
For any $\phi\in C^{\infty}_c(\Omega)$ after integrating by parts we have
\EQ{\label{eq2.6}
\langle b\cdot\nabla u, \phi\rangle=-\int_{\Omega} b\cdot(\nabla\phi )u\,dx.
}
Therefore we can apply H\"older inequality and get, in case $q\ge n$,
\Eq{
\int_{\Omega}|b\cdot\nabla u\phi |~dx\le \|b\|_{n}\|\phi\|_{\frac{nq'}{n-q'}}\|u\|_{1,q}\le C \|b\|_{n}\|\phi\|_{W^{1,q'}}\|u\|_{W^{1,q}},
}
and in case $1\le q<n$,
\Eq{
\int_{\Omega}|b\cdot\nabla\phi u|~dx\le \|b\|_{n}\|\phi\|_{W^{1,q'}}\|u\|_{\frac{nq}{n-q}}\le C \|b\|_{n}\|\phi\|_{W^{1,q'}}\|u\|_{W^{1,q}}.
}
The above shows \eqref{eq2.4}.

To prove the second estimate \eqref{eq2.5}, we approximate $b$ by $b_k\in C^{\infty}(\Omega)$ with $\div b_k=0$ and $\|b-b_k\|_{n}\rightarrow0$ as $k\to\oo$. Then
\EQN{
\|b\cdot\nabla u\|_{W^{-1,q}(\Omega)} &\le \|(b_k- b)\cdot\nabla u\|_{W^{-1,q}(\Omega)} +\|b_k\cdot\nabla u\|_{W^{-1,q}(\Omega)}\\
&\le C \|b_k- b\|_n
\|u\|_{1,q} +C\|b_k\|_{L^{\infty}(\Omega)}\|u\|_{q} \\
&\le
\varepsilon\|u\|_{1,q}+C(\varepsilon,p,b,\Om)\|u\|_{q}.
}
We have used \eqref{eq2.4} for the second inequality, and chosen $k$ so that $C\|b_k- b\|_n \le \e$.
\end{proof}

\begin{lemma}\label{drift_est_weak}
Let $\Om$ be a bounded Lipschitz domain in $\R^n$, $n\ge 3$.
For any $b\in L^{n,\infty}(\Omega)$, $\div b=0$, and $u\in W^{1,q}(\Omega)$, $1\le q<\infty$, we have
\[
b\cdot\nabla u\in W^{-1,q}(\Omega)~\text{and}~\|b\cdot \nabla u\|_{W^{-1,q}(\Omega)}\le C_p\|b\|_{L^{n,\infty}(\Omega)}\|u\|_{W^{1,q}(\Omega)}
\]
\end{lemma}
\begin{proof}
We use a similar argument to Lemma \ref{drift_est}. For any $\phi\in C^{\infty}_c(\Omega)$ we still have \eqref{eq2.6}.
By H\"older and Sobolev inequalities in Lorentz spaces (Lemmas \ref{Holder} and \ref{Sobolev}), in case $q\ge n$ we have
\EQN{
\int\limits_{\Omega}|b\cdot\nabla u\phi |~dx &\le C\|b\|_{L^{n,\infty}(\Omega)}\|\phi\|_{L^{\frac{nq'}{n-q'},q'}(\Omega)}\|u\|_{1,q} \\
&\le C\|b\|_{L^{n,\infty}(\Omega)}\|\phi\|_{W^{1,q'}(\Omega)}\|u\|_{W^{1,q}(\Omega)}.
}
The case $1\le q< n$ is done similarly:
\EQN{
\int\limits_{\Omega}|b\cdot\nabla\phi u|~dx &\le C\|b\|_{L^{n,\infty}(\Omega)}\|\phi\|_{W^{1,q'}(\Omega)}\|u\|_{L^{\frac{nq}{n-q},q}(\Omega)}\\
&\le C\|b\|_{L^{n,\infty}(\Omega)}\|\phi\|_{W^{1,q'}(\Omega)}\|u\|_{W^{1,q}(\Omega)} . \qedhere
}
\end{proof}

\begin{remark}
In Lemmas \ref{drift_est} and \ref{drift_est_weak} we restrict to $n\ge3$. When $n=2$,
$b u$ may not be in $L^2$ when $u \in W^{1,2}$ even if $b \in L^2$; see \cite[Remark 2.4]{Filonov-Shilkin}.
\end{remark}
\begin{lemma}[Pressure]\label{pressure}
Let $\Om$ be a bounded Lipschitz domain in $\R^n$, $n \ge 2$, and $1<q<\infty$. For any $f \in W^{-1,q}(\Om)^n$ such that
$
\bka{f,\zeta}=0$ for all $\zeta \in W^{1,q'}_{0,\si}(\Om)$,
there exists a unique $p \in L^q_0(\Om)$, such that $\norm{p}_{q} \le C \| f\|_{-1,q}$ and
\[
\bka{f,\zeta}=\int_\Om p \div \zeta, \quad \forall \zeta \in W^{1,q'}_{0}(\Om)^n.
\]
\end{lemma}
This is a special case of \cite[Theorem III.5.3]{Galdi} for bounded domains except the bound. The bound is because for any $g \in L^{q'}_0(\Om)$ with $\|g\|_{q'} \le 1$,
there is $\zeta\in W^{1,q'}_0(\Om)^n$ such that
$g=\div \zeta$ with $\norm{\zeta}_{W^{1,q'}} \le C \norm{g}_{q'}$  (by Bogovskii map), hence
$\|p\|_q = \sup _{g }\int pg \le \|f\|_{-1,q} \|\zeta\|_{1,q'} \le C \|f\|_{-1,q} $.

The following is concerned with the unperturbed Stokes system, i.e., \eqref{drifteq} with $b=0$:
\EQ{\label{Stokes}
-\Delta v +\nabla p=f,\quad \div v=0, \quad v|_{\pd \Om}=0.
}
We talk about solution pairs, $q$-weak solutions and very weak solutions of \eqref{Stokes} in the same way given by \eqref{soln-pair}, \eqref{weak-soln}, and \eqref{very-weak-soln} for \eqref{drifteq}.

\begin{theorem}[Existence, uniqueness and $L^q$-estimates]\label{th:GSS}
Let $1<q<\infty$ and let $\Om\subset\R^n$, $n \ge 2$, be
a bounded Lipschitz domain
with sufficiently small Lipschitz constant $L>0$ $($i.e., $L \le L_0$ where $L_0=L_0(n,q)>0)$.  Then for each given $f \in W^{-1,q}(\Om)^n$,
there is a unique solution pair $(v,p)\in W^{1,q}(\Om)^n \times L^{q}(\Om)$ of \eqref{Stokes} with $\int_\Om p=0$ and
\EQ{
\norm{v}_{1,q,\Omega}+\norm{p}_{q,\Omega}\le C_1 \norm{f}_{-1,q,\Omega},
}
for some $C_1=C_1(n,q,\Om)$.
\end{theorem}

The above theorem is due to Galdi, Simader, and Sohr \cite{GSS94}, which also considers nonzero $\div v$ and boundary value of $v$.
Such a result is well known for a more regular domain, and originated from the classical work of Cattabriga \cite{Cattabriga} for $C^2$-domains in $\R^3$. See the references of \cite{GSS94} and \cite[Theorem IV.6.1]{Galdi} for the many literatures in between.

\section{Proof of {Proposition \ref{main1} and Theorem \ref{main2}}}\label{sec3}
In this section we study the perturbed Stokes system \eqref{drifteq} and  prove {Proposition \ref{main1} and Theorem \ref{main2}}.

We first prove the existence and uniform a priori estimates for weak solutions of
\EQ{ \label{eq3.1}
-\Delta u+\lambda b\cdot\nabla u + \nb \pi =f,\quad \div u=0, \quad u|_{\pd \Om}=0,
}
in $\Om \subset \R^n$, $n \ge 3$, for $\la\in\R$. Recall that a weak solution $u$ belongs to $ W^{1,2}_{0,\sigma}(\Omega)$ and satisfies a weak form similar to \eqref{weak-soln}.

\begin{lemma}[Weak solutions]\label{energy}
Let $\Om$ be a bounded Lipschitz domain in $\R^n$, $n\ge 3$.
Assume $b\in L^{n,\infty}(\Omega)^n$ and $\div b=0$.
Then for any $\la \in \R$ and $f \in W^{-1,2}(\Omega)^n$,
there exists a unique weak solution $u\in W^{1,2}_{0,\sigma}(\Omega)$
of \eqref{eq3.1}. Furthermore, there is $\pi \in L^2(\Om)$ so that $(u,\pi)$ solves \eqref{eq3.1} in distributional sense, $\int_\Om \pi=0$, and
\EQ{\label{eq6.1}
\|u\|_{ W^{1,2}(\Omega)}+\|\pi\|_{L^2(\Omega)}\le C\|f\|_{W^{-1,2}(\Om)},
}
for some constant $C=C(n,\Om)$ independent of $\la$ and $b$.
\end{lemma}

\begin{proof} We may assume $\la=1$. Define bilinear form $\mathcal{B}: W^{1,2}_{0,\sigma}(\Om)\times W^{1,2}_{0,\sigma}(\Om)\rightarrow \Bbb R$ by
\EQ{
\mathcal{B}(u,v)=\int_{\Om}\nabla u:\nabla v+(b\cdot\nabla) u \cdot v\, dx.
}
The integral $\int (b\cdot\nabla) u \cdot v\, dx$ is defined by \eqref{eq2.1} with $q=2$ (which
uses Lemmas \ref{Holder} and \ref{Sobolev} and that $\Om$ is bounded and Lipschitz).
This bounds enables us to justify by approximation that
\EQ{\label{eq3.4}
\int_\Om (b\cdot \nb) u\cdot u =  \int_\Om b\cdot \nb(\frac12|u|^2) = 0,
}
using the condition $\div b=0$. Thus we get the coercivity of this bilinear form, indeed,
\EQ{
\mathcal{B}(u,u)=\|\nabla u\|_{L^2(\Om)}^2, \quad \forall u\in W_{0,\sigma}^{1,2}(\Om).
}
Therefore, by Lax-Milgram theorem, for any $f \in W^{-1,2}(\Omega)^n$, there exists a unique $u\in W_{0,\sigma}^{1,2}(\Om)$ such that
\EQ{
\mathcal{B}(u,v)=\langle f,v\rangle,\quad \forall v\in W_{0,\sigma}^{1,2}(\Om).
}
This is the weak form of \eqref{eq3.1} and hence
$u$ is the unique weak solution. By taking $v=u$, we get $\norm{u}_{W^{1,2}_0(\Om)} \le \norm{f}_{W^{-1,2}(\Om)}$.
Lastly, we apply Lemma \ref{pressure} with $q=2$ (and using that $\Om$ is bounded and Lipschitz) to get an associated pressure $\pi\in L^2_0(\Om)$ that satisfies the bound.
\end{proof}

We next derive a priori estimates in $W^{1,q}$-norm.
\begin{lemma}[A priori $W^{1,q}$-estimate when $q>2$] \label{apriori}
Assume $b\in L^{n,\oo}(\Omega)$, $\div b=0$, and $2<q<\infty$.
We further assume either
\EN{
\item [\textup{(i)}] $b\in L^n(\Omega)$, or
\item [\textup{(ii)}] $\|b\|_{n,\oo}\le \e_1$ for some $\e_1(n,q,\Om)>0$  sufficiently small.
}
Then there exists a constant $C$ such that for any $0\le\lambda\le1$,
any $f\in W^{-1,q}(\Omega)$,
and any $q$-weak solution $u$ of \eqref{eq3.1}, there is  $\pi \in L^q_0(\Om)$ so that $(u,\pi)$ is a solution pair of \eqref{eq3.1}
and
we have
\[
\|u\|_{W^{1,q}(\Omega)} + \norm{\pi}_{L^q(\Om)}\le C\|f\|_{W^{-1,q}(\Omega)}.
\]
\end{lemma}

The constant $C=C(n,q,\Om)$ does not depend on other properties of $\lambda,u$, and $f$, except those specified.

\begin{proof}
Let $\phi= f+\Delta u-\lambda b\cdot\nabla u $. It is a distribution in $W^{-1,q}(\Om)^n$ by Lemmas \ref{drift_est} and \ref{drift_est_weak}, and it
vanishes on $W^{1,q'}_{0,\si}(\Om)$ by the weak form (similar to \eqref{weak-soln}) of \eqref{eq3.1}. Hence $\phi$ is given by a pressure gradient, $\phi=\nb \pi$ with $\pi \in L^q_0(\Om)$ by Lemma \ref{pressure}, and $(u,\pi)$ is a solution pair of \eqref{eq3.1}.
By Theorem \ref{th:GSS},
\Eq{
\|u\|_{W^{1,q}(\Omega)} + \| \pi \|_{L^q(\Om)}
\le C_1\|f-\lambda b\cdot\nabla u\|_{W^{-1,q}(\Omega)}\le C_1\|f\|_{W^{-1,q}(\Omega)}+C_1\|b\cdot \nabla u\|_{W^{-1,q}(\Omega)}.
}
For the drift term, if $ b \in L^n(\Om)$, we can apply Lemma \ref{drift_est} and get that
\[
\|b\cdot \nabla u\|_{W^{-1,q}(\Omega)}\le
\varepsilon \|u\|_{1,q;\Omega}+C \|u\|_{q,\Omega}.
\]
Taking $\e = (4C_1)^{-1}$ and using $\|u\|_q \le \de\|u\|_{1,q}+ C(\de) \|u\|_2$, we get
\EQ{\label{3.5}
\|u\|_{W^{1,q}(\Omega)}+ \| \pi \|_{L^q(\Om)}\le
C_1\|f\|_{-1,q}+\frac{1}{4}\|u\|_{1,q}+C\|u\|_{q}\le C_1\|f\|_{-1,q}+\frac{1}{2}\|u\|_{1,q}+C\|u\|_{2}.
}
If $b \in L^{n,\infty}(\Om)$,  we apply Lemma \ref{drift_est_weak} and assume $\norm{b}_{n,\infty} \le\e/C_q $. We still get \eqref{3.5}.

As $q>2$, we have $u \in W^{1.2}(\Om)$ and
$\|u\|_{W^{1,2}(\Omega)}\le C(n,q,\Omega)\|f\|_{-1,q}$ by Lemma \ref{energy}. Using this uniform bound in $W^{1,2}(\Omega)$ we get from \eqref{3.5} that
\[
\|u\|_{W^{1,q}(\Omega)} + \| \pi \|_{L^q(\Om)} \le C\|f\|_{-1,q}. \qedhere
\]
\end{proof}

\begin{proof}[Proof of {Proposition \ref{main1} and Theorem \ref{main2}} when $q \ge 2$]

By Lemmas \ref{drift_est} and \ref{drift_est_weak}, the linear differential operators
\[
L_\la (u,\pi) = -\Delta u+\lambda b\cdot\nabla u + \nb \pi,\quad 0 \le \la \le 1,
\]
are bounded linear maps from $\mathfrak X_q$ to  $\mathfrak Y_q$, where
\EQ{\label{XqYq}
\mathfrak X_q:= W^{1,q}_{0,\si}(\Om) \times L^q_0(\Om) ,\quad  \mathfrak Y_q:=W^{-1,q}(\Om)^n.
}
By Lemma \ref{apriori},
they satisfy $\norm{x}_{\mathfrak X_q} \le C \norm{L_\la x}_{\mathfrak Y_q}$ for all $\la \in [0,1]$ and $x \in \mathfrak X$, with a uniform constant $C$. As $L_0: \mathfrak X_q \to  \mathfrak Y_q$ is onto by Theorem \ref{th:GSS}, the map $L_1$ is also onto by the method of continuity \cite[Theorem 5.2]{GT}. This
means that for every $f =\div \GG \in \mathfrak Y_q$ %
there is a solution $(u,\pi)\in \mathfrak X_q$ of \eqref{drifteq}, and
shows {Proposition \ref{main1} and Theorem \ref{main2}} when $q \ge 2$.
\end{proof}

\begin{lemma}[A priori $W^{1,q}$-estimate when $q<2$]\label{secondapriori}
The same statement of Lemma \ref{apriori} remains true when $1<q<2$.
\end{lemma}
\begin{proof}
The proof for $q<2$ is by duality. Let $u$ be a given $q$-weak solution of \eqref{eq3.1} with source $\GG \in L^{q}(\Omega)^{n\times n}$.
Take any $\FF \in L^{q'}(\Omega)^{n\times n}$. By {Proposition \ref{main1} and Theorem \ref{main2}} when $q \ge 2$, there is  solution pair $(v,\pi)\in \mathfrak X_{q'}$ (see \eqref{XqYq}) to the following equation
\EQ{\label{eq3.7}
-\Delta v-b\cdot\nabla v+\nabla \pi =\div \FF,
}
with $\norm{\nb v}_{q'} \le C \norm{\FF}_{q'}$.
We can use $u$ as a test function for \eqref{eq3.7} and get that
\EQ{
\int_{\Omega}\nabla v:\nabla u-b\cdot\nabla v \cdot u=-\int_{\Omega} \FF:\nabla u.
}
For the left hand side we can apply equation \eqref{eq3.1} for $u$ with test function $v$ and get that
\EQN{
\int_{\Omega}\FF:\nabla u&=
\int_{\Omega}\GG:\nabla v
\\
&\le \| \GG\|_q \| \nb v\|_{q'} \le \| \GG\|_q C\| \FF\|_{q'}.
}
Since $\FF \in L^{q'}$ is arbitrary, this implies a-priori bound of $u$: $\|\nb u\|_{q}\le C\|\GG\|_{q}$.
\end{proof}

\begin{proof}[Proof of {Proposition \ref{main1} and Theorem \ref{main2}} when $q< 2$]

It is by the same proof for the case $q>2$, but we use the a priori estimate Lemma \ref{secondapriori} instead of Lemma \ref{apriori}.
\end{proof}

\section{Wolf's local pressure projection}\label{sec4}

In this section, we will consider Wolf's local pressure projection \cite{Wolf15, Wolf17,JWZ}, which will be useful in our proof of Theorem \ref{main3}.
Of particular interest to us is to show a uniform bound of the map in the $L^2$-setting on  $\Om_{x_0,R}$, $x_0 \in \pd \Om$, that only depends on the Lipschitz constant of the boundary, but not on other properties of the boundary graph function. For this purpose we first prepare a particular Bogovskii map.

\begin{lemma}[Bogovskii map]\label{Bogmap}
Let $n \ge2$ and $\Om_R = \{ (x',x_n) \in B_R(0) \subset \R^n, x_n>\ga(x')\}$, where $\ga(x')$ is a Lipschitz function defined for $x'\in B_R'(0) \subset \R^{n-1}$ with Lipschitz constant $L \in (0,1/2)$, and $\ga(0)=0$.

\textup{(a)}\ There is a constant $\e = \e(L)\in (0,1/2)$,
independent of $R$, and its dependence on $\ga$ is only through the constant $L$,
such that $\Om_R $ is star-shaped with respect to any point in the ball $B\subset\Om_R$ centered at
$(0, (1-\e)R)$ with radius $\e R$.

\textup{(b)}\ There is a linear map $\Phi$ that maps any scalar $f \in L^q_0(\Om_R)$, $1<q<\infty$, to a vector field $v = \Phi f \in W^{1,q}_0(\Omega_R)^n$ and
\EQ{\label{eq4.22}
\div v= f, \quad \norm{\nb v}_{L^q(\Omega_R)} \le  C_1\norm{f}_{L^q(\Omega_R)},
}
for some constant $C_1=C_1(n,L,q)>0$ which is independent of $R$ and its dependence on $\ga(x')$ is only through the constant $L$.
The map $\Phi$ is independent of $q$ for $f \in C^\infty_c(\Omega_R)$.
\end{lemma}

\begin{proof} (a) Let $x=(x',x_n)$ be any point in $\Om_R$, and $y=(y',y_n)$ be any point in $B$. We want to show $(1-t)y + tx \in \Om_R$ for any $ t \in (0,1)$. We have $(1-t)y + tx \in B_R$ as $B_R$ is convex. It remains to show $(1-t)y_n + tx_n > \ga((1-t)y '+ tx' )$. Indeed,
\EQN{
&(1-t)y_n + tx_n - \ga((1-t)y '+ tx' )\\
& >(1-t)(1-2\e) R + t\ga(x') - [\ga(x') + L (1-t)|y'-x'|]
\\
&\ge (1-t)\bket{ (1-2\e)R  -LR - L(1+\e)R },
}
which is nonnegative
if $\e \le \frac{1-2L}{2+L}$.

\medskip

(b) We may assume $R=1$ by rescaling $\td f(x) = f(Rx)$ and $\td v(x) = R^{-1}v(Rx)$, noting that \eqref{eq4.22} and the Lipschitz constant $L$ are scaling invariant.
Let $\Phi\in C^\infty_c(\R^n)$ be supported in $|x|\le \frac 12$ with
$\int_{\R^n} \Phi=1$. Let $\e$ and $B$ be as in Part (a), and $\phi(x) =\e^{-n} \Phi(\e^{-1}(x', x_n - 1+\e))$, which
is supported in $B$. By Part (a),  $\Om_R$ is star-shaped with respect to any point in the ball $B$.
We use the Bogovskii formula for $f \in C^\infty_c(\Omega_R)$ with $\int_\Om f=0$:
\EQ{\label{eq4.23}
v(x)=\int_{\Om}  N(x,y)f(y)\,dy, \quad N(x,y)=\frac {x-y}{|x-y|^n} \int_{|x-y|}^\infty \phi\bke{y+r \frac {x-y}{|x-y|}}r^{n-1} dr.
}
By the properties of the Bogovskii formula (see \cite[Lemma III.3.1]{Galdi}), we have $v \in W^{1,q}_0(\Omega_R)$, $\div v= f$ and $\norm{\nb v}_{L^q} \le C_1 \norm{f}_{L^q}$.
Note that formula \eqref{eq4.23} does not depend on the choice of $\ga(x')$, hence nor does the constant $C_1$.
It shows \eqref{eq4.22}.
\end{proof}

If one assumes Part (a), then \eqref{eq4.22} follows from \cite[Lemma III.3.1]{Galdi}. We prove Part (b) to highlight that the Bogovskii formula \eqref{eq4.23} and the constant $C_1$ do not depend on $\ga(x')$.

\bigskip

We now consider Wolf's local pressure projection \cite{Wolf15, Wolf17,JWZ}.
In a general bounded domain $\Om_0 $ in $\R^n$, for $1<q<\infty$, we consider the problem
\EQ{\label{wolfequation}\left\{\gathered
-\Delta v+\nabla\pi=f,\quad \div v=0,
\\
v|_{\partial\Omega_0}=0,\quad \int_{\Omega_0}\pi\,dx=0,
\endgathered
\right.
}
for $f \in W^{-1,q}(\Omega_0)$. If there is a unique solution $(v,\pi)\in W^{1,q}_0(\Om_0) \cap L^q_0(\Om_0)$, we define the projection
\EQ{
\mathcal{W}_{q,\Omega_0}: W^{-1,q}(\Omega_0)\rightarrow W^{-1,q}(\Omega_0), \quad
\mathcal{W}_{q,\Omega_0}(f)=\nabla\pi.
}
It is automatic that $\mathcal{W}_{q,\Omega_0}(\nabla\pi)=\nabla\pi$.
By Theorem \ref{th:GSS} (due to Galdi, Simader and Sohr \cite{GSS94}), this projection operator is well defined for $1<q<\infty$ if $\Om_0$ is a bounded Lipschitz domain with sufficiently small Lipschitz constant $L \le L_0(n,q)$. By
Brown and Shen \cite{BS95}, it is defined for bounded Lipschitz domains in $\R^3$ with arbitrary Lipschitz constants and $q \in (q_0',q_0)$, $q_0=3+\e$, for some $\e>0$.

For our application to the proof of Theorem \ref{main3}, we will only consider $q=2$ which requires less regularity of $\Om_0$. This is important to us because our $\Om_0$ is either a ball inside $\Om$, or
$\Omega_0=\Omega\cap B_r(x_0)$ for $x_0\in\pd {\Omega}$ and $r>0$, which is in general not Lipschitz unless the Lipschitz constant $L$ of $\pd \Omega$ in $B_r(x_0)$ is sufficiently small.
If $L$ is large, then $\Omega_0$ may contain cusps. For example, consider the intersection of the unit disk and the region $x_n > \phi(|x'|)$, with $\phi(t)=\frac 52 \sqrt2 t^2 - 3 \sqrt 2 t^4$ for $|t|<\frac {\sqrt 2} 2$ and $\phi(t)=\sqrt 2 - |t|,$ for $|t|\ge\frac {\sqrt 2}2$.

\begin{theorem}\label{theoremL2wolf}
Assume $p=2$ and either $\Omega_0=B_r(x_0)\subset\Omega$, or $\Omega_0=\Om_{x_0,r} =\Omega\cap B_r(x_0)$ for $x_0\in\pd{\Omega}$, $r>0$, is the subset of $\Om$ above a Lipschitz function  $\ga(x')$ (under suitable coordinate rotation and translation) with Lipschitz constant $L\in(0,1/2)$, as described in Lemma \ref{Bogmap}.  Let $f=\div \FF$ for some $\FF\in L^2(\Om_0)^{n\times n}$. Then the function $\pi\in L^2_0(\Om_0)$  defined as $\mathcal{W}_{2,\Omega_0}(f)=\nabla\pi$ satisfies the following inequality
\EQ{\label{pWolfbound}
\|\pi\|_{2,\Om_0}\le c_1\|\FF\|_{L^2(\Om_0)},
}
where $c_1=c_1(n,L)$ is independent of $r$, and its dependence on $\ga(x')$ is only through the constant $L$.
\end{theorem}

\begin{proof}
By \cite[Theorem IV.1.1 and Remark IV.1.4]{Galdi}, there exists a unique  pair $(v,\pi) \in W^{1,2}_0(\Om_0)^n\times L^2_0(\Om_0)$ satisfying
\EQ{
-\Delta v+\nabla\pi =\div \FF,\quad \div v=0,\quad v|_{\partial\Omega_0}=0,\quad\int_{\Om_0}\pi=0.
}
We take $v$ as a test function to get that
\EQ{\label{vL2bound}
\|\nb v\|_{L^{2}(\Omega_0)}\le \|\FF\|_{L^2(\Omega_0)}.
}
Here the constant is equal to one and does not depend on the domain. Next, since $\int_{\Om_0}\pi=0$, we can apply the Bogovskii map to construct a function $\phi_0\in W^{1,2}_0(\Omega_0)^n$
\EQ{\label{phi_0}
\div\phi_0=\pi, \quad
}
with
the following $L^2$ estimate
\EQ{\label{phibound}
\|\nabla\phi_0\|_{L^2(\Om_0)}\le c\|\pi\|_{L^2(\Om_0)},
}
where $c=c(n,L)$ is independent of $r$ or $\gamma(x')$. Indeed, in the internal case it follows from the property of the Bogovski map in a ball, and in the boundary case it follows from Lemma \ref{Bogmap}. We now use $\phi_0$ as a test function in \eqref{wolfequation} and get that
\EQ{
\int_{\Om_0}\nabla v:\nabla\phi_0-\pi\div\phi_0~dx=-\int_{\Om_0}\FF: \nabla\phi_0~dx.
}
From \eqref{phi_0} and \eqref{phibound} we get that
\EQN{
\|\pi\|_{L^2(\Om_0)}^2&\le \|\nabla\phi_0\|_{L^2(\Om_0)}(\|\FF\|_{L^2(\Om_0)}+\|\nabla v\|_{L^2(\Om_0)})
\\&\le c\|\pi\|_{L^2(\Om_0)}(\|\FF\|_{L^2(\Om_0)}+\|\nabla v\|_{L^2(\Om_0)}).
}
Therefore from \eqref{vL2bound} we get the pressure estimate
\EQ{
\|\pi\|_{2,\Omega_0}\le c_1 \|\FF\|_{L^2(\Omega_0)},
}
with $c_1=2c(n,L)$.
\end{proof}

\begin{remark}\label{Wolfpdecom}
For our application, we consider a weak solution $(u,\pi)$ of \eqref{drifteq}
in a Lipschitz domain $\Omega$ with $b\in L^{n,\infty}(\Omega)$, $\div b=0$, $\GG\in L^2(\Omega)$, and let $\Om_0$ be a subset of $\Om$ as in Theorem \ref{theoremL2wolf} (either $\Omega_0=B_r(x_0)\subset\Omega$, or $\Omega_0=\Om_{x_0,r}$ for $x_0\in\pd{\Omega}$).  We may apply Wolf's local pressure projection in $\Om_0$ to entire  \eqref{drifteq} to get
\EQ{
\nb \pi = \cW_{2,\Om_0}(\nb \pi ) =\cW_{2,\Om_0}[\Delta u-\div(b\otimes u)+\div \GG].
}
Thus we have a local pressure decomposition $\pi-(\pi)_{\Om_0} =\pi_1+\pi_2+\pi_3$ in $\Om_0$, where $\pi_1,\pi_2,\pi_3\in L^2_0(\Om_0)$ are given by
\EQ{
\nabla\pi_1=\mathcal{W}_{2,\Om_0}(\Delta u),\quad \nabla\pi_2=\mathcal{W}_{2,\Om_0}(-\div(b\otimes u)),\quad \nabla\pi_3=\mathcal{W}_{2,\Om_0}(\div \GG).
}
By Theorem  \ref{theoremL2wolf}, we have for $c_1=c_1(n,L)$
\EQ{\label{eq4.14}
\norm{\pi-(\pi)_{\Om_0}}_{L^2(\Om_0)} \le c_1\bke{ \norm{\nb u}_{L^2(\Om_0)}
+ \norm{b\otimes u}_{L^2(\Om_0)} + \norm{\GG}_{L^2(\Om_0)}}.
}
\end{remark}

\section{A priori bound for large drift and $p$ close to $2$}\label{sec5}

In this section we prove the key a priori bound, which will be used to prove Theorem \ref{main3}.

\begin{theorem}\label{Capriori}
Let $\Omega$ be a bounded Lipschitz domain in $\R^n$, $n\ge3$, with Lipschitz constant $L\in(0,1/2)$. Let $b\in L^{n,\infty}(\Omega)^n$, $\div b=0$, and $(u,\pi)$ be a weak solution pair of \eqref{drifteq}. Then there exists $p_0=p_0(\Omega,\|b\|_{L^{n,\infty}(\Omega)})>2$ such that for any $p\in(2,p_0]$ and $\GG\in L^{p}(\Omega)^{n\times n}$, the weak solution $u\in W^{1,2}_{0,\si}(\Omega)$ is in fact in $ W^{1,p}(\Omega)^n$ and
\EQ{\label{eq5.1}
\|u\|_{W^{1,p}(\Omega)}\le c\|\GG\|_{L^p(\Omega)}
}
for some constant $c(\Omega,\|b\|_{L^{n,\infty}(\Omega)})>0$.
\end{theorem}

\begin{proof}
For $x\in \overline \Om$ and $R>0$, denote $\Om_{x,R}=\Om \cap B(x,R)$.
Because $\pd \Om$ is Lipschitz and compact, there is a finite number of points $x_i\in \pd \Om$, $i=1,\ldots,N$, and a radius $R_0\in (0,1]$, such that $\pd \Om$ is contained in $\cup_{i=1}^N \Om_{x_i,R_0}$, and
\[
\Om_{x_i,2R_0}= \{ (x',x_n) \in B_{2R_0}(0) \subset \R^n,\ x_n>\ga_i(x')\},
\]
after suitable coordinate rotation and translation, where $\ga_i(x')$ is a Lipschitz function defined for $x'\in B_{2R_0}'(0) \subset \R^{n-1}$ with Lipschitz constant $L$, and $\ga_i(0)=0$, for $i=1,\ldots,N$.

To prove \eqref{eq5.1}, we first consider the interior case and take $B_{\rho}=B_{\rho}(x_0)\subset\Omega$.
Let $\bar{u}=u-k$ where $k$ is an arbitrary real constant that we will choose later.  By Remark \ref{Wolfpdecom}, we can decompose $\pi - (\pi)_{B_\rho}=\pi_1+\pi_2+\pi_3$ in $B_{\rho}$, where $\pi_1,\pi_2,\pi_3$ are given by
\EQ{
\nabla\pi_1=\mathcal{W}_{2,B_{\rho}}(\Delta \bar{u}),\quad \nabla\pi_2=\mathcal{W}_{2,B_{\rho}}(-\div(b\otimes \bar{u})),\quad \nabla\pi_3=\mathcal{W}_{2,B_{\rho}}(\div \GG).
}
By Theorem \ref{theoremL2wolf},
we have the following bound for some global constant $c_1$
\EQ{\label{wolfbound}
\|\pi_1\|_{2,B_{\rho}}\le c_1\|\nabla \bar{u}\|_{2,B_{\rho}},\quad\|\pi_2\|_{2,B_{\rho}}\le c_1 \|b\bar{u}\|_{2,B_{\rho}},\quad \|\pi_3\|_{2,B_{\rho}}\le c_1\|\GG\|_{2,B_{\rho}}.
}
We may replace $(u,\pi)$ by $(\bar u,\pi-(\pi)_{B_\rho})$ in \eqref{drifteq} and multiply the equation by $\eta^4 \bar{u}$ where $\eta$ is a smooth cutoff function on $B_{\rho}$, with $\eta=1$ in $B_r$, $0<r<\rho$. Integrating, we get
\EQS{\label{3.1}
\int_{B_{\rho}} |\nabla (\eta^2\bar{u})|^2
&\le \int_{B_{\rho}} |\bar{u}|^2|\nabla \eta^2|^2+\int_{B_{\rho}} |\bar{u}|^2|b||\nabla \eta|\eta^3\\
&\quad +\int_{B_{\rho}} (|\pi_1|+|\pi_2|+|\pi_3|)|\bar{u}||\nabla\eta|\eta^3+\langle\div \GG, \bar{u}\eta^4\rangle_{B_{\rho}}.
}
Next we estimate each of the terms on the right hand side separately:
\EQS{
\int_{B_{\rho}} |\bar{u}|^2|b||\nabla \eta|\eta^3
&\le \frac{c}{\rho-r} \|b\|_{L^{n,\infty}}\|\bar{u}\eta\|_{2,B_{\rho}}\|\bar{u}\eta^2\|_{L^{\frac{2n}{n-2},2}(B_{\rho})}
\\
&\le \frac{c}{(\rho-r)^2} \|b\|_{L^{n,\infty}(B_{\rho})}^2\|\bar{u}\eta\|^2_{2,B_{\rho}}+\frac{1}{16}\|\bar{u}\eta^2\|_{W^{1,2}(B_{\rho})}^2.
}
Pressure terms are estimated in the following way
\EQN{
\int_{B_{\rho}}|\pi_1||\bar{u}||\nabla\eta|\eta^3 &\le \frac{c}{\rho-r}\|\pi_1\|_{2,B_{\rho}}\|\eta^2 \bar{u}\|_{2,B_{\rho}}
\\
&\le
\frac{c}{\rho-r}\|\nabla \bar{u}\|_{2,B_{\rho}}\|\eta^2 \bar{u}\|_{2,B_{\rho}}\le  \frac{1}{16}\|\nabla \bar{u}\|_{2,B_{\rho}}^2+\frac{c}{(\rho-r)^2}\|\bar{u}\eta^2\|^2_{2,B_{\rho}},
}
\EQN{
\int_{B_{\rho}}|\pi_2||\bar{u}||\nabla\eta|\eta^3 &\le \|\pi_2\|_{2,B_{\rho}}\|\eta^2\bar{u}\|_{2,B_{\rho}}\le c\|\bar{u}\|_{2,B_{\rho}}\|b\|_{L^{n,\infty}(B_{\rho})}\|\bar{u}\|_{L^{\frac{2n}{n-2},2}(B_{\rho})}
\\
&\le C(b)\|\bar{u}\|_{2,B_{\rho}}^2+\frac{1}{16}\|\bar{u}\|^2_{W^{1,2}(B_{\rho})},
}
\EQN{
\int_{B_{\rho}}|\pi_3||\bar{u}||\nabla\eta|\eta^3 &\le \frac{c}{\rho-r}\|\pi_3\|_{2,B_{\rho}}\|\eta^2 \bar{u}\|_{2,B_{\rho}}\le  \frac{c}{\rho-r}\|\GG\|_{2,B_{\rho}}\|\eta^2 \bar{u}\|_{2,B_{\rho}}
\\
&\le \frac 14 \|\GG\|_{2,B_{\rho}}^2+\frac{c}{(\rho-r)^2}\|\eta^2 \bar{u}\|_{2,B_{\rho}}^2.
}
Next we estimate the term from the right hand side of \eqref{drifteq},
\EQN{
\langle\div \GG, \bar{u}\eta^4\rangle_{B_{\rho}} &=-\int_{B_{\rho}} \GG :\nabla(\bar{u}\eta^4)=-\int_{B_{\rho}} \GG: \nabla(\eta^2\bar{u})\eta^2-2\int_{B_{\rho}} \GG :\bar{u}\nabla\eta\eta^3
\\
&\le \frac{1}{16}\|\nabla(\eta^2\bar{u})\|^2_{2,B_{\rho}}+\frac{C}{(\rho-r)^2}\|\eta^2 \bar{u}\|_{2,B_{\rho}}^2+ \frac 14\|\GG\|_{2,B_{\rho}}^2.
}
Combining these estimates with \eqref{3.1} we get the following (with $c(b)$ depending on $b$ only through $\norm{b}_{L^{n,\infty}(\Om)}$)
\EQ{
\int |\nabla (\eta^2\bar{u})|^2\le \frac{1}{4}\|\bar{u}\|^2_{W^{1,2}(B_{\rho})}+\frac{c(b)}{(\rho-r)^2}\|\bar{u}\|_{2,B_{\rho}}^2+\|\GG\|_{2,B_{\rho}}^2.
}
Finally, we get the inequality for all $0<r<\rho\le R_0 \le 1$,
\EQ{
\int_{B_r} |\nabla \bar{u}|^2\le \frac{1}{4}\|\nabla \bar{u}\|^2_{2,B_{\rho}}+\frac{c(b)}{(\rho-r)^2}\|\bar{u}\|_{2,B_{\rho}}^2+\|\GG\|_{2,B_{\rho}}^2.
}
Here we apply Lemma 3.1 from Giaquinta \cite[page 161]{Gia} to remove the first term on the right side:
\EQ{
\int_{B_r} |\nabla \bar{u}|^2\le c \bke{\frac{c(b)}{(\rho-r)^2}\|\bar{u}\|_{2,B_{\rho}}^2+\|\GG\|_{2,B_{\rho}}^2}.
}
By taking $\rho=2r \le R_0$, we get the following inequality
\EQ{
\int_{B_r} |\nabla \bar{u}|^2\le \frac{c(b)}{r^2}\|\bar{u}\|_{2,B_{2r}}^2+c\|\GG\|_{2,B_{2r}}^2.
}
Since $\bar{u}=u-k$ and constant $k$ was arbitrary we can choose $k=(u)_{B_{2r}}$ and apply Poincar\'e inequality to get
\EQ{\label{5.12}
\frac{1}{|B_r|}\int_{B_r} |\nabla \bar{u}|^2
\le c(b) \Big(\frac{1}{|B_{2r}|}\int_{B_{2r}}|\nabla u|^{\frac{2n}{n+2}}\Big)^{\frac{n+2}{n}}+\frac{c}{|B_{2r}|}\|\GG\|_{2,B_{2r}}^2.
}

Next we consider the boundary case, firstly for $\Omega_{\rho}=\Omega\cap B_{\rho}(x_0)$ with $x_0$ being one of $x_i\in\pd{\Omega}$, $i=1,\ldots,N$, $0<\rho<R_0/6$. We do similar pressure decomposition $\pi-(\pi)_{\Om_\rho}=\pi_1+\pi_2+\pi_3$ where
\EQ{
\nabla\pi_1=\mathcal{W}_{2,\Omega_{\rho}}(\Delta u),\quad \nabla\pi_2=\mathcal{W}_{2,\Omega_{\rho}}(-\div(b\otimes u)),\quad \nabla\pi_3=\mathcal{W}_{2,\Omega_{\rho}}(\div \GG).
}
By Theorem \ref{theoremL2wolf}, the local pressure projection  $\mathcal{W}_{2,\Omega_{\rho}}$ is well defined, and we have the bound \eqref{pWolfbound} with constant $c_1$ independent of $\rho$. Take $r\le \rho/2$ and $\eta$ is a cutoff function on $B_{\rho}(x_0)$ such that $\eta=1$ on $B_r(x_0)$. We extend $\GG,b,u$ as well as all parts of the pressure $\pi_k,~k=1,2,3$ with $0$ to the region $B_{\rho}(x_0)\setminus\Omega$. Since $u$ satisfies the boundary condition $u|_{\partial\Omega}=0$, the function $\eta^4u$ is still an admissible test function for \eqref{drifteq} with $\pi$ replaced by $\pi-(\pi)_{\Om_\rho}$. Therefore we get similar to \eqref{3.1} that
\EQS{
\int_{B_{\rho}} |\nabla (\eta^2u)|^2
&\le \int_{B_{\rho}} |u|^2|\nabla \eta^2|^2+\int_{B_{\rho}} |u|^2|b||\nabla \eta|\eta^3\\
&\quad +\int_{B_{\rho}} (|\pi_1|+|\pi_2|+|\pi_3|)|u||\nabla\eta|\eta^3-\int_{B_{\rho}} \GG :\nabla(u\eta^4).
}
Due to our zero extension of pressure we have the following estimates
\EQS{
\|\pi_1\|_{2,B_{\rho}}= \|\pi_1\|_{2,\Om_{\rho}}\le c\|\nabla u\|_{2,\Omega_{\rho}}=c\|\nabla u\|_{2,B_{\rho}}
\\
\|\pi_2\|_{2,B_{\rho}}=  \|\pi_1\|_{2,\Om_{\rho}}\le c\|bu\|_{2,\Omega_{\rho}}=c\|b u\|_{2,B_{\rho}}
\\
\|\pi_3\|_{2,B_{\rho}}=  \|\pi_1\|_{2,\Om_{\rho}}\le c\|\GG\|_{2,\Omega_{\rho}}=c\|\GG\|_{2,B_{\rho}}.
}
Therefore, we can apply similar approach as in the internal case and get %
\EQS{\label{5.16}
\frac{1}{|B_r|}\int_{B_r} |\nabla u|^2 &\le \frac{c(b)}{r^2}\frac{1}{|B_{2r}|}\|u\|_{2,B_{2r}}^2+\frac{c}{|B_{2r}|}\|\GG\|_{2,B_{2r}}^2
\\
&\le c(b) \Big(\frac{1}{|B_{2r}|}\int_{B_{2r}}|\nabla u|^{\frac{2n}{n+2}}\Big)^{\frac{n+2}{n}}+\frac{c}{|B_{2r}|}\|\GG\|_{2,B_{2r}}^2.
}
To get the second inequality we have used Poincare inequality, unlike the internal case, using $u=0$ on the set $\Om^c\cap B_{\rho}(x_0)$ of non zero measure, with the ratio $\frac{|\Om^c\cap B_{\rho}(x_0)|}{|B_{\rho}(x_0)|}>c$ bounded from below by a constant $c=c(n,L)$ independent of $\rho$.

Lastly, for arbitrary ball $B_{2r}(x_0)$ with $B_{2r}(x_0)\cap\partial\Om\neq\emptyset$, there exists $y_0\in B_{2r}(x_0)\cap\partial\Om$. Then
\EQ{
B_r(x_0)\subset B_{3r}(y_0)\subset B_{6r}(y_0)\subset B_{8r}(x_0).
}
Using \eqref{5.16} for $B_{3r}(y_0)$ we get
\EQS{\label{5.17}
\frac{1}{|B_r|}\int_{B_r(x_0)} |\nabla u|^2
&\le \frac{C}{|B_{3r}|}\int_{B_{3r}(y_0)} |\nabla u|^2
\\
&\le c(b) \Big(\frac{1}{|B_{6r}|}\int_{B_{6r}(y_0)}|\nabla u|^{\frac{2n}{n+2}}\Big)^{\frac{n+2}{n}}+\frac{C}{|B_{6r}|}\|\GG\|_{2,B_{6r}(y_0)}^2
\\
& \le c(b) \Big(\frac{1}{|B_{8r}|}\int_{B_{8r}(x_0)}|\nabla u|^{\frac{2n}{n+2}}\Big)^{\frac{n+2}{n}}+\frac{C}{|B_{8r}|}\|\GG\|_{2,B_{8r}(x_0)}^2.
}

Combining the internal estimate \eqref{5.12} and the boundary estimate \eqref{5.17}, we conclude
\EQ{\label{5.19}
\frac{1}{|B_r|}\int_{B_r(x_0)} |\nabla u|^2 \le
c(b) \Big(\frac{1}{|B_{8r}|}\int_{B_{8r}(x_0)}|\nabla u|^{\frac{2n}{n+2}}\Big)^{\frac{n+2}{n}}+\frac{C}{|B_{8r}|}\|\GG\|_{2,B_{8r}(x_0)}^2,
}
for any $x_0$ in a big cube containing $\overline \Om$, and any $r\le R_0/6$.
We can now apply Proposition 1.1 of Giaquinta \cite[page 122]{Gia} (also see \cite[Proposition 3.7]{DongKim}) to get that $\nb u \in L^p_\loc$ for $p \in (2,p_0)$ for some $p_0 >2$, and
\EQ{\label{0619}
\bke{\fint_{B_r(x_0)} |\nabla u|^p}^{\frac 1p} \le
c_1\bke{\fint_{B_{8r}(x_0)} |\nabla u|^2}^{\frac 12} + c_1\bke{\fint_{B_{8r}(x_0)} |\GG|^p}^{\frac 1p},
}
for all $x_0 \in \overline \Om$ and $r\le R_0/6$, with
$p_0$ and $c_1$ depending only on $n$, $c(b)$, and $R_0$. Summing \eqref{0619} over a finite cover of $\overline \Om$ of balls of radius $R_0/6$, and using the a priori $W^{1,2}$-estimate in Lemma \ref{energy}, we get \eqref{eq5.1}.
\end{proof}

\begin{remark}
Tracking the proof, we find that the constant $p_0$ only depends on $n$, $L$, $\norm{b}_{L^{n,\infty}}$ and $R_0$, while the constant $c$ in \eqref{eq5.1} also depends on diam $\Om$.
\end{remark}

\section{Proof of Theorem \ref{main3}}\label{sec6}
We now prove Theorem \ref{main3} on the unique existence of $W^{1,q}$ solutions for $q\in (p_0',p_0)$ sufficiently close to $2$, when $b\in L^{n,\infty}(\Om)$ is not small.

\begin{proof}[Proof of Theorem \ref{main3}]

First we consider case $q>2$. Let $p_0(\Om,\|b\|_{L^{n,\infty}(\Om)})>2$ be the constant in Theorem \ref{Capriori} for given $\Om$ and $b$. We may assume $p_0<n$ (so that we can use \eqref{eq2.1} for $q \le p_0$).
Take any $q\in (2,p_0)$ and $\GG\in L^q(\Om)^{n\times n}$. By Lemma \ref{energy} we know that there exists a weak solution pair  $u\in W^{1,2}(\Omega)$ and $\pi\in L^2_0(\Omega)$ of \eqref{drifteq}. Next we apply  Theorem \ref{Capriori} and get that
\EQ{
\|u\|_{W^{1,q}(\Om)}\le C(b)\|\GG\|_{L^q(\Om)}.
}
This bound also implies uniqueness of the solutions. Since $u\in W^{1,q}$, $b\in L^{n,\infty}$, and $\div b=0$, we get that $\GG+b\otimes u\in L^q(\Om)$ by \eqref{eq2.1}, and we can apply Lemma \ref{pressure} to get that the pressure $\pi$ is indeed in $L^q_0(\Om)$ and
\[
\|\pi\|_{L^q(\Om)}\le C(b)\|\GG\|_{L^q(\Om)}.
\]
This proves the case $q>2$.

For case $q<2$ we will use a duality argument. First, we will prove a priori estimate in $W^{1,q}$ for $q\in(p_0',2)$. Suppose $u$ is a $q$-weak solution of \eqref{drifteq} with right hand side $\div\GG$. For any $\GG'  \in L^{q'}(\Omega)^{n \times n}$, $2<q'<p_0$, let $v$ be the unique $q'$-weak solution of the following system
\EQ{\label{drifteq_test}
-\Delta v -b\cdot\nabla v+\nabla \pi=\div \GG',\quad \div v=0, \quad v|_{\pd \Om}=0,
}
Therefore, we have an a-priori bound
\EQ{
\|v\|_{W^{1,q'}(\Om)}\le C(b)\|\GG'\|_{L^{q'}(\Om)},
}
guaranteed by the first part of this theorem. Next we use $u$ as a test function in equation \eqref{drifteq} for function $v$ and integrate by part to get that
\EQS{
\int_{\Om}\GG' :\nabla u\,dx&=-\int_{\Om}(\nabla v+bv):\nabla u \, dx=-\int_{\Om}(\nabla v-bu):\nabla v \, dx
\\
&=\int_{\Om} \GG:\nabla v\,dx\le \|\GG\|_{L^{q}(\Om)}\|\nabla v\|_{L^{q'}(\Om)}\le C(b)  \|\GG\|_{L^{q}(\Om)}\|\GG'\|_{L^{q'}(\Om)}.
}
Since the matrix function $\GG'\in L^{q'}(\Om)$ is arbitrary we get that
\EQ{\label{smallparpriori}
\|u\|_{W^{1,q}(\Om)}\le C(b)  \|\GG\|_{L^{q}(\Om)}.
}

To prove existence we use approximation approach. For $\GG\in L^q(\Om)$, choose a sequence  $\GG_n\in L^2(\Om)$ such that $\GG_n\rightarrow\GG$ in $L^q(\Om)$. There exists a sequence of weak $W^{1,2}$ solutions $u_n$ of \eqref{drifteq} with source term $\div \GG_n$ instead of $\div\GG$ and they satisfy \eqref{smallparpriori},
\[
\|u_n\|_{W^{1,q}(\Om)} \le C \|\GG_n\|_{L^{q}(\Om)} \le C\|\GG\|_{L^{q}(\Om)} .
\]
 Therefore there exists $u\in W^{1,q}(\Om)$ and a subsequence $u_n\rightharpoonup u$ in $W^{1,q}(\Om)$ weakly, with
$\|u\|_{W^{1,q}(\Om)} \le C\|\GG\|_{L^{q}(\Om)} $.
The last step is passing to the limit and constructing pressure. For any divergence free test function $\zeta\in C^{\infty}_{c,\sigma}(\Om)$, we have the weak form \eqref{weak-soln} for $u_n$,
\EQ{
\int_{\Om} (\nb u_n - b \otimes u_n): \nb \zeta   = -\int_\Om \GG_n : \nb \zeta.
}
We can pass to the limit $n\rightarrow\infty$ and get that
\EQ{
\int_{\Om} (\nb u - b \otimes u): \nb \zeta   = -\int_\Om \GG : \nb \zeta.
}
Hence $u$ is a $q$-weak solution of \eqref{drifteq} with right side $\div \GG$ satisfying the weak form \eqref{weak-soln}, uniqueness follows from the a priori estimate, similarly to the case $q>2$.
Lastly we apply Lemma \ref{pressure} to prove that there exists pressure $\pi\in L^q_0(\Om)$ and
\EQ{
\|\pi\|_{L^q(\Om)}\le C(b)  \|\GG\|_{L^{q}(\Om)}.
}
This finishes the proof of Theorem \ref{main3}.
\end{proof}

\section{Proof of Proposition \ref{main4}}\label{sec7}
In this section we provide a scheme of the proof of Proposition \ref{main4}, which is for the scalar equations \eqref{scalareq-conv} and \eqref{scalareq-div}.  We first state an $W^{1,2}$-result.

\begin{lemma}[Weak solutions for scalar equations with critical drifts]\label{scalar-weak-solution}
Let $\Om$ be a bounded Lipschitz domain in $\R^n$, $n\ge 3$.
Assume $b\in L^{n,\infty}(\Omega)^n$ and $\div b\ge 0$.
Then for any $F ,G\in L^2(\Omega)$,
there exist a unique weak solution $v\in W^{1,2}_{0}(\Omega)$
of  \eqref{scalareq-conv} with right hand side $G$, and a unique weak solution $u\in W^{1,2}_{0}(\Omega)$
of  \eqref{scalareq-div} with right hand side $F$. Moreover,
\EQ{
\|v\|_{ W^{1,2}(\Omega)}\le C\|G\|_{L^2(\Om)},\quad
\|u\|_{ W^{1,2}(\Omega)}\le C\|F\|_{L^2(\Om)},
}
for some constant $C=C(n,\Om)$ independent of $b$.
\end{lemma}

The proof is similar to that for Lemma \ref{energy} for the vector equation \eqref{drifteq}. We do not need to deal with the pressure, and we replace \eqref{eq3.4}  by
\EQ{
-\int_\Om (b\cdot \nb) v\cdot v = - \int_\Om b\cdot \nb(\frac12|v|^2)%
\ge 0,\quad
\int_\Om \div (ub) u= -\int_\Om ub\cdot \nb u   \ge 0,
}
which keep the bilinear forms coercive.

Next we will prove an a-priori estimate in $W^{1,q}_0(\Om)$ for \eqref{scalareq-conv} with $2 < q < p_0$,
 which is similar to Theorem \ref{Capriori}.
Let $v$ be a $q$-weak solution of \eqref{scalareq-conv} with right hand side $G$.
We first consider the interior case and take $B_{\rho} =B_{\rho}(x_0)\subset\Omega$. Let $0<r<\rho$. Similarly take $\bar{v}=v-k$ where $k$ is an arbitrary real constant that we will choose later. Next we multiply the equation by $\eta^4 \bar{v}$ where $\eta$ is a smooth cutoff function on $B_{\rho}$,
with $\eta=1$ in $B_r$. Due to our assumption on drift $\div b\ge 0$ we need to estimate the drift term
\EQS{\label{7.0}
\int_{B_{\rho}} (b\cdot\nabla \bar v)\bar v\eta^4~dx=\frac{1}{2}\int_{B_{\rho}} b\cdot\nabla (\bar v^2)\eta^4~dx\le -\frac{1}{2}\int_{B_{\rho}} b\bar v^2\nabla\eta^4~dx.
}
Therefore we get a similar energy inequality as in Theorem \ref{Capriori} but without pressure
\EQ{\label{7.1}
\int_{B_{\rho}} |\nabla (\eta^2\bar{v})|^2
\le \int_{B_{\rho}} |\bar{v}|^2|\nabla \eta^2|^2+2\int_{B_{\rho}} |\bar{v}|^2|b||\nabla \eta|\eta^3
+\int_{B_{\rho}}\div G \cdot \bar{v}\eta^4.
}
Next we use the same estimates as in Theorem \ref{Capriori} to remove $\frac 14\int_{B_{\rho}} |\nabla \bar{v}|^2$ from the right side, then take $\rho=2r$
 and choose $k=(v)_{B_{\rho}}$  to get the following
\EQ{\label{7.2}
\frac{1}{|B_r|}\int_{B_r} |\nabla v|^2
\le c(b) \Big(\frac{1}{|B_{2r}|}\int_{B_{2r}}|\nabla v|^{\frac{2n}{n+2}}\Big)^{\frac{n+2}{n}}+\frac{c}{|B_{2r}|}\|G\|_{2,B_{2r}}^2.
}
The proof of the boundary case is done in a similar matter as Theorem  \ref{Capriori}: we do zero extension of function $v$ outside of $\Om$, and prove \eqref{7.2} for $x_0 \in \pd \Om$ and $r<R_0/12$. Finally we prove \eqref{7.2} for general $x_0 \in \overline \Om$ with radius $2r$ on the right side replaced by $8r$. By Proposition 1.1 of Giaquinta \cite[page 122]{Gia},
there exists $p_0(\Om,\|b\|_{L^{n,\infty}(\Om)})>2$ such that for $q\in(2,p_0),~\nabla v\in L^q$ and
\EQ{\label{7.3}
\|v\|_{W^{1,q}(\Omega)}\le c\|G\|_{L^q(\Omega)}.
}
This is our desired a priori estimate.

\medskip

Finally we can prove Proposition \ref{main4}.
\begin{proof}[Proof of Proposition \ref{main4}]
 We follow the same steps of the proof of Theorem \ref{main3}.  Let $p_0>2$ be the constant  from our a priori estimate \eqref{7.3}. Take any $q\in (2,p_0)$ and $G\in L^q(\Om)$. From $L^2$ theory (Lemma \ref{scalar-weak-solution}) we know that there exists a weak solution  $v\in W^{1,2}(\Omega)$ of \eqref{scalareq-conv}. Next we apply \eqref{7.3} and get that
\EQ{
\|v\|_{W^{1,q}(\Om)}\le C(b)\|G\|_{L^q(\Om)}.
}
This finishes the proof of (a).

For case (b) we will use a duality argument. First, we will prove an a priori estimate in $W^{1,q}$ for $q\in(p_0',2)$. Suppose $u$ is a $q$-weak solution of \eqref{scalareq-div} with right hand side $F$. For any $G  \in L^{q'}(\Omega)^{n}$, $2<q'<p_0$, let $v$ be the unique $q'$-weak solution of the dual system \eqref{scalareq-conv} with right hand side $G$ and
\EQ{
\|v\|_{W^{1,q'}(\Om)}\le C(b)\|G\|_{L^{q'}(\Om)},
}
guaranteed by the first part of this theorem. We use $u$ as a test function in equation \eqref{scalareq-conv} and use  \eqref{scalareq-div} to get that
\EQS{
\int_{\Om}G\cdot\nabla u\,dx&=-\int_{\Om}(\nabla v\cdot\nabla u-(b\nabla v) u )\, dx
\\
&=\int_{\Om} F\cdot\nabla v\,dx\le \|F\|_{L^{q}(\Om)}\|\nabla v\|_{L^{q'}(\Om)}\le C(b)  \|F\|_{L^{q}(\Om)}\|G\|_{L^{q'}(\Om)}.
}
Since function $G\in L^{q'}(\Om)$ is arbitrary we get that
\EQ{\label{scalarapri}
\|u\|_{W^{1,q}(\Om)}\le C(b)  \|F\|_{L^{q}(\Om)}.
}

Finally, we prove existence using the approximation approach: For $F\in L^q(\Om)$, choose a sequence  $F_n\in L^2(\Om)$ such that $F_n\rightarrow F$ in $L^q(\Om)$. By Lemma \ref{scalar-weak-solution}, there exists a sequence of weak $W^{1,2}$ solutions $u_n$ of \eqref{scalareq-div} with right side $F_n$ and they satisfy \eqref{scalarapri},
\[
\|u_n\|_{W^{1,q}(\Om)} \le C \|F_n\|_{L^{q}(\Om)} \le C\|F\|_{L^{q}(\Om)} .
\]
 Therefore there exists $u\in W^{1,q}(\Om)$ and a subsequence $u_n\rightharpoonup u$ in $W^{1,q}(\Om)$ weakly, with
$\|u\|_{W^{1,q}(\Om)} \le C\|F\|_{L^{q}(\Om)} $.
The last step is passing to the limit. For any smooth test function $\zeta\in C^{\infty}_{c}(\Om)$, we have the weak form of \eqref{scalareq-div} for $u_n$,
\EQ{
\int_{\Om} (\nb u_n - b u_n)\cdot \nb \zeta   = -\int_\Om F_n : \nb \zeta.
}
We can pass to the limit $n\rightarrow\infty$ to show that $u$ satisfies the weak form of \eqref{scalareq-div}. This gives (b).
\end{proof}
\section*{Acknowledgments}
We warmly thank Hongjie Dong, Hyunseok Kim, Hyunwoo Kwon and Tuoc Phan for very helpful comments and references. The research of both MC and TT was partially supported by Natural Sciences and Engineering Research Council of Canada (NSERC) under grant RGPIN-2023-04534.


\addcontentsline{toc}{section}{\protect\numberline{}{\hspace{2mm}References}}

\medskip

Misha Chernobai, Max Planck Institute for Mathematics in the Sciences, Leipzig  Inselstraße 22, 04103, Germany;
e-mail: mchernobay@gmail.com

\medskip

Tai-Peng Tsai, Department of Mathematics, University of British Columbia,
Vancouver, BC V6T 1Z2, Canada; e-mail: ttsai@math.ubc.ca

\end{document}